\documentclass[12pt]{amsart}

\usepackage{fullpage,amssymb}

\usepackage{mathrsfs}
\usepackage{multirow}
\usepackage{amsmath}
\usepackage{comment}
\usepackage{tikz}
\usepackage{mathtools}
\usepackage{arydshln}
\usepackage{graphicx}
\usepackage{hyperref}

\newtheorem{theorem}{Theorem}[section]
\newtheorem{lemma}[theorem]{Lemma}
\newtheorem{corollary}[theorem]{Corollary}
\newtheorem{problem}[theorem]{Problem}
\newtheorem{proposition}[theorem]{Proposition}
\newtheorem*{notation*}{Notation}
\newtheorem*{p*}{Proposition~\ref{h.s.o.p}}
\theoremstyle{definition}
\newtheorem{definition}[theorem]{Definition}

\newtheorem{remark}[theorem]{Remark}

\newcommand{\M}{{\operatorname{Mat}}}

\newcommand{\C}{{\mathbb C}}

\newcommand{\Z}{{\mathbb Z}}

\newcommand{\K}{K}

\newcommand{\SL}{{\operatorname{SL}}}

\renewcommand{\det}{\operatorname{det}}
\newcommand{\per}{\operatorname{per}}

\newcommand{\sgn}{\operatorname{sgn}}

\newcommand{\Hom}{\operatorname{Hom}}
\newcommand{\kar}{\operatorname{char}}

\newcommand{\rk}{\operatorname{rk}}
\newcommand{\trk}{\operatorname{trk}}
\newcommand{\brk}{\operatorname{brk}}
\newcommand{\symRank}{\operatorname{trk_S}}
\newcommand{\symBRank}{\operatorname{brk_S}}

\newcommand{\sdet} {\operatorname{sdet}}
\newcommand{\sperm} {\operatorname{sper}}
\newcommand{\perm} {\operatorname{perm}}

\newcommand{\tensorsOfRank}[1]{\mathcal{Z}_{#1}}

\newcommand{\id}{\operatorname{id}}
\newcommand{\defeq}{\coloneqq}

\newcommand{\symPower}[2]{\operatorname{Sym}^{#1}{#2}}

\newcommand{\todo}[1]{\textcolor{red}{TODO: #1}}

\title{On the tensor rank of $3\times 3$ permanent and determinant}
\author{Siddharth Krishna}
\address{Department of Computer Science, New York University}
\email{siddharth@cs.nyu.edu}

\author{Visu Makam}
\address{Department of Mathematics, University of Michigan, Ann Arbor}
\email{visu@umich.edu}

\thanks{The second author was supported by NSF grant DMS-1601229.}

\begin{document}
\maketitle

\begin{abstract}
  The tensor rank and border rank of the $3 \times 3$ determinant tensor is known to be $5$ if characteristic is not two. In this paper, we show that the tensor rank remains $5$ for fields of characteristic two as well. We also include an analysis of $5 \times 5$ and $7 \times 7$ determinant and permanent tensors, as well as the symmetric $3 \times 3$ permanent and determinant tensors. We end with some remarks on binary tensors.
\end{abstract}
 
\section{Introduction}

An alternate way of looking at the rank of a matrix $A$ is as the smallest integer $r$ such that you can write $A$ as a sum of $r$ rank $1$ matrices.
The definition of tensor rank is a generalization of this idea. We consider the tensor space $V_1 \otimes V_2 \otimes \dots \otimes V_n$, where $V_i$ denote finite dimensional vector spaces over a field $K$. Tensors of the form $v_1 \otimes \dots \otimes v_n$ with $v_i \in V_i$ are called simple (or rank $1$) tensors. 

\begin{definition}
The tensor rank $\trk(T)$ of a tensor $T \in V_1 \otimes V_2 \otimes \dots \otimes V_n$ is defined as the smallest integer $r$ such that $T = T_1 + \dots + T_r$ for simple tensors $T_i$. 
\end{definition}

Let $\tensorsOfRank{r}$ denote the set of tensors of rank $\leq r$. Unfortunately, $\tensorsOfRank{r}$ is not Zariski-closed, giving rise to the notion of border rank.

\begin{definition}
The border rank $\brk(T)$ of a tensor $T \in  V_1 \otimes V_2 \otimes \dots \otimes V_n$ is the smallest integer $r$ such that $T \in \overline{\tensorsOfRank{r}}$
\end{definition}

The tensor and border rank of certain tensors have been well studied due to their connections with computational complexity. For instance, the multiplication of an $m \times n$ matrix with an $n \times l$ matrix is a bilinear map, and can hence be seen as a tensor in $\M_{n,m} \otimes \M_{l,n} \otimes \M_{m,l}$, where $\M_{p,q}$ is the space of $p \times q$ matrices.
Every expression for this matrix multiplication tensor in terms of simple tensors gives rise to an algorithm for matrix multiplication, and hence the rank of this tensor is a measure of the complexity of the problem. For example, the tensor and border rank of the matrix multiplication tensor for $m = n = l = 2$ is $7$, and the complexity of the resulting algorithm (Strassen's algorithm) is $O(n^{\log_2 7})$. Various improvements have since been made, see \cite{ACT,Lbook}.

In this paper, we study the determinant and permanent tensors. Let $\{e_i\ |\ 1\leq i \leq n\}$ denote the standard basis for $K^n$, and let $\Sigma_n$ denote the symmetric group on $n$ letters. The determinant tensor is 
$$
\det_n = \sum_{\sigma \in \Sigma_n} \sgn(\sigma) e_{\sigma(1)} \otimes e_{\sigma(2)} \otimes \dots \otimes e_{\sigma(n)} \in (K^n)^{\otimes n}
$$
where $\sgn(\sigma)$ is the sign of the permutation $\sigma$.

Similarly, the permanent tensor is defined as
$$ 
\per_n = \sum_{\sigma \in \Sigma_n}  e_{\sigma(1)} \otimes e_{\sigma(2)} \otimes \dots \otimes e_{\sigma(n)} \in (K^n)^{\otimes n}.
$$

The determinant and permanent tensors have been studied before, see \cite{Derksen} for known upper and lower bounds. For $K = \C$, the tensor rank of $\det_3$ and $\per_3$ were precisely determined by Ilten and Teitler in \cite{IT} to be $5$ and $4$ respectively. This is done by analyzing certain Fano schemes parametrizing linear subspaces contained in the
hypersurfaces $\per_3 = 0$ and $\det_3 = 0$. Using simpler linear algebraic techniques, Derksen and Makam~\cite{DM-explicit} show that the border rank (and tensor rank) of $\det_3$ and $\per_3$ are $5$ and $4$ respectively. Moreover, they show that this holds for all fields of characteristic not equal to two.
Observe that in characteristic $2$, the determinant and permanent tensors are equal. In this paper, we remove the dependence on the characteristic of the field for the tensor rank of the determinant. The main result of this paper is the following:

\begin{theorem} \label{main}
For any field $K$, the tensor rank of $\det_3$ is $5$.
\end{theorem}

This allows us to extend a result of Derksen in \cite{Derksen} to arbitrary characteristic.

\begin{corollary}
For any field $K$, we have $\brk(\det_n) \leq \trk(\det_n) \leq \left(\frac{5}{6}\right)^{\lfloor n/3 \rfloor} n!$.
\end{corollary}

\subsection{Organization}
In Section~\ref{sec-lower-bounds} we present the standard methods for obtaining lower bounds on tensor and border rank, specifically flattenings and Koszul flattenings.
In Section~\ref{sec-3x3}, we study $\det_3$ and $\per_3$ and show upper bounds via computer-aided search and lower bounds via case analysis and flattenings.
We then study the $5 \times 5$ and $7 \times 7$ determinant and permanent tensors in Section~\ref{sec-5x5-7x7}, and the symmetric determinant and permanent tensors in Section~\ref{sec-symmetric}.
Finally, we turn our attention to binary tensors, and study how the rank of binary tensors varies as you change the characteristic of the underlying field in Section~\ref{sec-binary}.

\section{Lower Bounds}
\label{sec-lower-bounds}

In this section, we recount the standard techniques for showing lower bounds on the tensor and border rank of tensors.
We first describe Strassen's Theorem for deriving lower bounds, and then we introduce flattenings in order to give a slight generalization of the result. Finally, we describe Koszul flattenings. The generalization of Strassen's Theorem that we present is a special case of Koszul flattenings.

\subsection{Flattenings}

The core argument to obtain lower bounds is the following.
Suppose $f$ is a polynomial that vanishes on $\tensorsOfRank{r}$, the set of all tensors of tensor rank $\leq r$. Now, if $f(T) \neq 0$ for some tensor $T$, then we can deduce that $\trk(T) > r$. In fact we even have $\brk(T) > r$. This is because the zero set of $f$ is a Zariski closed set containing $\tensorsOfRank{r}$ and hence contains $\overline{\tensorsOfRank{r}}$. Hence $T \notin \overline{\tensorsOfRank{r}}$, which means $\brk(T) > r$.

Thus we can lower bound the tensor rank or border rank of $T$ if we can find polynomials that vanish on $\tensorsOfRank{r}$. It turns out that it is difficult to find these polynomials in general for large $r$. One of the first non-trivial results in this direction was given by Strassen on so-called $3$-slice tensors, see \cite{ACT}.

\begin{theorem} [Strassen] \label{3-slice}
Let $T  = (e_1 \otimes A + e_2 \otimes B + e_3 \otimes C) \in \K^3 \otimes K^m \otimes K^m$ for $A,B, C \in K^m \otimes K^m$.
Viewing $K^m \otimes K^m$ as $\M_{m,m}$, if $A$ is invertible, then
$$
\brk(T) \geq m + \frac{1}{2}\rk(BA^{-1}C - CA^{-1}B).
$$
\end{theorem}

In essence, Strassen's Theorem says for any tensor $T$ as above, if $k$ is the rank of $BA^{-1}C - CA^{-1}B$, then the $k \times k$ minors of $BA^{-1}C - CA^{-1}B$ vanish on tensors of border rank less than $ m+ \lceil k/2 \rceil$.
A more modern way to view the above theorem is as follows.

\begin{proposition} \label{skew-symm}
Let $T,A,B,C$ be as in Theorem~\ref{3-slice}, then
$$ 
\brk(T) \geq \frac{1}{2}\rk \begin{pmatrix} 0 & A & B \\ -A & 0 & C \\ -B & -C & 0 \end{pmatrix}.
$$
\end{proposition}

When $A$ is invertible, the following (block) Gaussian elimination procedure shows that we recover Strassen's result:

\begin{align*}
\begin{pmatrix} 0 & A & B \\ -A & 0 & C \\ -B & -C & 0 \end{pmatrix} &\xrightarrow[R_1 \mapsto A^{-1}R_1]{R_2 \mapsto A^{-1}R_2} \begin{pmatrix} 0 & I & A^{-1}B \\ -I & 0 & A^{-1}C \\ -B & -C & 0 \end{pmatrix}  \\
  &\xrightarrow[R_3 \mapsto R_3 + CR_1 -BR_2]{C_3 \mapsto C_3 - C_2(A^{-1}B) + C_1(A^{-1}C) } \begin{pmatrix} 0 & I & 0 \\ -I & 0 & 0 \\ 0 & 0  & CA^{-1}B - BA^{-1}C \end{pmatrix}.
\end{align*}

We remark here that while doing Gaussian elimination of block matrices, one can add a left multiplied block row to other block rows and a right multiplied block column to other block columns (see \cite{DM}). The above Proposition is a generalization of Strassen's result, because it doesn't require $A$ to be invertible

Proposition~\ref{skew-symm} is a special case of a flattening. A flattening of a tensor space $V = V_1 \otimes \dots \otimes V_n$ is any linear map $V \rightarrow \M_{n,n}$. The following straightforward proposition shows how flattenings can be used to show explicit lower bounds on the border rank of tensors.

\begin{proposition} \label{flattening}
Let $\phi: V_1 \otimes V_2 \otimes \dots \otimes V_n \rightarrow \M_{m,m}$ be a linear map. Suppose that for all $S \in \tensorsOfRank{1}$ we have $\rk(\phi(S)) \leq r$, then for any tensor $T \in V_1 \otimes V_2 \otimes \dots \otimes V_n$, we have
\[\brk(T) \geq \frac{\rk (\phi(T))} {r}.\]
\end{proposition}

We can now prove Proposition~\ref{skew-symm}.
\begin{proof} [Proof of Proposition~\ref{skew-symm}]
Consider $\phi:K^3 \otimes K^m \otimes K^m \rightarrow \M_{3m,3m}$ where 
$$\phi(e_1 \otimes A + e_2 \otimes B + e_3 \otimes C) = \begin{pmatrix} 0 & A & B \\ -A & 0 & C \\ -B & -C & 0 \end{pmatrix}.
$$

Note that for any $S \in \tensorsOfRank{1}$ we have $\rk(\phi(S)) \leq 2$. Now, apply Proposition~\ref{flattening}.
\end{proof}

\subsection{Koszul flattenings}

While it is difficult to find flattenings that give non-trivial lower bounds, one class of flattenings that have proven useful are Koszul flattenings.

Landsberg has constructed some explicit tensors in $\C^m \otimes \C^m \otimes \C^m$ of border rank at least $2m-2$ (resp. $2m-4$) for $m$ even (resp. $m$ odd), see \cite{Landsberg}. The technique to show lower bounds was to use Koszul flattenings. In \cite{DM-ncrk}, the case of $m$ odd is improved from $2m-4$ to $2m-3$ using a concavity result from \cite{DM}, and in \cite{DM-explicit}, these results are extended to an arbitrary field.

To describe Koszul flattenings, we first construct a linear map as follows.
Let $m = 2p+1$ and $n$ be positive integers, and $V$ and $W$ be $m$ and $n$-dimensional vector spaces over $K$ respectively. Let $\bigwedge^i(V)$ denote the $i^{th}$ exterior power of the vector space $V$. We have a linear map $L: V \rightarrow \Hom(\bigwedge^pV,\bigwedge^{p+1}V)$ given by $L(v) : \omega \rightarrow v \wedge \omega$. 

For the tensor space $V \otimes W^{\otimes 2k}$, we define the Koszul flattening $\varphi: V \otimes W^{\otimes 2k} \rightarrow \M_{D,D}$ where $D = {m \choose p}n^{k}$ in the following way. First we have the map
\[V \otimes W^{\otimes 2k} = V \otimes W^{\otimes k} \otimes W^{\otimes k} \xrightarrow[L \otimes \id \otimes \id]{} \Hom(\bigwedge^pV,\bigwedge^{p+1}V) \otimes (W^{\otimes k} \otimes W^{\otimes k}).\]

Identifying $W^{\otimes k} \otimes W^{\otimes k}$ with $\M_{n^k,n^k}$, we get a map 
\[\Hom(\bigwedge^pV,\bigwedge^{p+1}V) \otimes (W^{\otimes k} \otimes W^{\otimes k}) = \M_{{m \choose p},{m \choose p}} \otimes \M_{n^k,n^k} \rightarrow \M_{D,D},\]
where the last map is given by the Kronecker product of matrices. Composing the above maps, we get $\varphi:V \otimes W^{\otimes 2k} \rightarrow \M_{D,D}$.

Before we can analyze the effect of the Koszul flattening on simple tensors, note the following property of $L$, see for example \cite{DM-ncrk}.
\begin{lemma}
For all $v \in V$, $\rk(L(v)) = {2p \choose p}$.
\end{lemma}

We can now show how the Koszul flattening gives us lower bounds on tensor rank.
\begin{lemma} \label{koszul-rank-one}
For any $S \in \tensorsOfRank{1}$, $\rk(\varphi(S)) \leq {2p \choose p}$. 
\end{lemma}

\begin{proof}
By the above lemma, we see that $L \otimes \id \otimes \id$ takes a rank $1$ tensor to a tensor of rank ${2p \choose p}$. Since all subsequent maps do not increase tensor rank, we get the required conclusion. In fact, a more careful analysis will show that $\rk(\varphi(S)) = {2p \choose p}$ for any simple tensor $S$, but we will not need it. 
\end{proof}

\begin{proposition} \label{general-koszul}
For a tensor $T \in V \otimes W^{\otimes 2k}$, we have 
\[\brk(T) \geq \frac{\rk (\varphi(T))}{{2p\choose p}}.\]
\end{proposition}

\begin{proof}
This follows from the above Lemma~\ref{koszul-rank-one} and Proposition~\ref{flattening}.
\end{proof}

The flattening in Proposition~\ref{skew-symm} is a special case of a Koszul flattening for $p = 1$ and $k = 1$.

\section{Tensor rank of the $3 \times 3$ determinant tensor}
\label{sec-3x3}

In the first part of this section, we first consider the known upper bounds for the ranks of $\det_3$ and $\per_3$, and then give explicit expressions for them that hold in any characteristic.
We then use the flattening-based techniques to give matching lower bounds in the second part.

\subsection{Upper bounds} \label{upper bounds}

An explicit expression for a tensor $T$ in terms of simple tensors naturally gives us an upper bound for tensor rank and border rank of $T$.
Glynn's formula (see \cite{Glynn}) for the permanent tensor is 

$$
\per_n = \frac{1}{2^{n-1}} \sum_{ v \in \{\pm 1\}^{n-1}} (e_1 + v_1 e_2 + \dots v_{n-1}e_n)^{\otimes n}.
$$

In particular, this shows that 
$$
\brk(\per_n) \leq \trk(\per_n) \leq 2^{n-1}
$$
as long as characteristic is not two. For the determinant tensor, known upper bounds are much weaker. The best known upper bound comes from Derksen's formula (see \cite{Derksen}) for $\det_3$.

\begin{align*}
    \det_3 = \frac{1}{2}\Big(& (e_3+e_2)\otimes (e_1-e_2)\otimes (e_1+e_2) \\
           &  {} + (e_1+e_2)\otimes (e_2-e_3)\otimes (e_2+e_3) \\
           &  {} + 2e_2\otimes (e_3-e_1)\otimes (e_3+e_1) \\
           &  {} + (e_3-e_2)\otimes (e_2+e_1)\otimes (e_2-e_1)  \\
           &  {} + (e_1-e_2)\otimes (e_3+e_2)\otimes (e_3-e_2) \Big).
\end{align*}

Derksen uses this expression along with Laplace expansions to show 
$$
\brk(\det_n) \leq \trk(\det_n) \leq \left(\frac{5}{6}\right)^{\lfloor n/3 \rfloor} n!.
$$

Unfortunately, both Glynn's and Derksen's expressions fail in characteristic two because they have denominators which are multiples of two. Hence the only known upper bound for the tensor rank of $\det_3$ and $\per_3$ was $6$, given by the defining expression. 

We find expressions for both $\det_3$ and $\per_3$ as a sum of $5$ simple tensors that are valid over any field $K$. To do this, we first consider $\det_3$ over $F_2$, the field of two elements (note that $\det_3 = \per_3$ over $F_2$). With the help of a computer, we found a way to write $\det_3 = \per_3$ as a sum of $5$ simple tensors. Then by carefully choosing the signs, we were able to find expressions for both $\det_3$ and $\per_3$ that work for any field.
We have 

\begin{align*}
\det_3 &= (e_2 + e_3) \otimes e_1 \otimes e_2 \\ 
       & \quad - (e_1 + e_3) \otimes e_2 \otimes e_1 \\
       & \quad - e_2 \otimes (e_1 + e_3) \otimes (e_2 + e_3)  \\
       & \quad + (e_2 - e_1) \otimes e_3 \otimes (e_1 + e_2 + e_3) \\
       & \quad + e_1 \otimes (e_2 + e_3) \otimes (e_1 + e_3),
\end{align*}

and

 \begin{align*}
   \per_3 &= (e_2 + e_3) \otimes e_1 \otimes e_2  \\
          & \quad + (e_1 + e_3) \otimes e_2 \otimes e_1 \\
          & \quad + e_2 \otimes (e_1 + e_3) \otimes (e_3 - e_2) \\
          & \quad + (e_1 + e_2) \otimes e_3 \otimes (e_1 + e_2 - e_3) \\
          & \quad + e_1 \otimes (e_2 + e_3) \otimes (e_3 - e_1).
\end{align*}

\begin{corollary}
$\brk(\det_3) \leq \trk(\det_3) \leq 5$. 
\end{corollary}

\begin{remark}
Tensor rank over $\Z$ is in general an undecidable problem, see \cite{Shitov}. However, the expressions above show that the tensor rank over $Z$ of both $\det_3$ and $\per_3$ is $\leq 5$. On the other hand, Theorem~\ref{main} shows that the tensor rank over $Z$ cannot be less than $5$, and so $\trk_{\Z} (\det_3) = \trk_Z(\per_3) = 5$.
\end{remark}

\subsection{Lower Bounds}
 
In every characteristic other than two, a direct application of Proposition~\ref{skew-symm} gives us that $\rk(\det_3) \geq 5$, see \cite{DM-explicit}. Let us recall the determinant tensor

$$
\det_3 = \sum_{\sigma \in \Sigma_3} \sgn(\sigma) e_{\sigma(1)} \otimes e_{\sigma(2)} \otimes e_{\sigma(3)} \in K^3 \otimes K^3 \otimes K^3.
$$

Identifying $K^3 \otimes K^3$ with $\M_{3,3}$ via $e_i \otimes e_j \mapsto E_{i,j}$, we can identify $K^3 \otimes (K^3 \otimes K^3)$ with $K^3 \otimes \M_{3,3}$. Under this identifcation, we have 

$$
\det_3 = e_1 \otimes \begin{pmatrix} 0 & 0 & 0 \\ 0 & 0& 1 \\ 0 & -1 & 0 \end{pmatrix} + e_2 \otimes \begin{pmatrix} 0 & 0& -1 \\ 0 & 0 & 0  \\ 1 & 0 & 0 \end{pmatrix} + e_3 \otimes \begin{pmatrix} 0 & 1 & 0 \\ -1 & 0 & 0  \\ 0 & 0 & 0 \end{pmatrix}
$$

We recall the proof of the following proposition from \cite{DM-explicit}, as we will modify the proof to remove the dependence on characteristic.

\begin{proposition} [\cite{DM-explicit}] \label{charnot2}
We have $\trk(\det_3) = \brk(\det_3) = 5$ if $\kar K \neq 2$. 
\end{proposition}

\begin{proof}
Applying Proposition~\ref{skew-symm}, we get that

$$\brk(\det_3) \geq \frac{1}{2} \rk \left( \arraycolsep=5pt\def\arraystretch{1} \begin{array}{ccc|ccc|ccc}  
0 & 0 & 0 & 0 & 0 & 0 & 0 & 0 & {\color{red} -1} \\
0 & 0 & 0 & 0 & 0 & {\color{red} 1} & 0 & 0 & 0 \\
0 & 0 & 0 & 0 & -1 & 0 & 1 & 0 & 0 \\
\hline
0 & 0 & 0 & 0 & 0 & 0 &0 & {\color{red} 1} & 0  \\
0 & 0 & -1 & 0 & 0 & 0 &-1 & 0 & 0 \\
0 & {\color{red} 1} & 0 &0 & 0 & 0 & 0 &0 & 0 \\
\hline
0 & 0 & 1 & 0 & -1& 0  & 0 & 0 & 0 \\
0 & 0 & 0  &{\color{red}  1} & 0 & 0 & 0 & 0 & 0 \\
{\color{red} -1} & 0 & 0 & 0 & 0 & 0 & 0 & 0 & 0 \\
\end{array} \right)
$$

This matrix contains only 12 nonzero entries of the form $\pm 1$. Six of these entries (marked red) are in a column or a row with no other nonzero entry, reducing our computation to a $3 \times 3$ minor $ \begin{pmatrix} 0 & -1& 1 \\ -1 & 0 & -1  \\ 1 & -1 & 0 \end{pmatrix}$. This minor has rank $3$ as long as characteristic is not two, and hence we have $\brk(\det_3) \geq \frac{9}{2} = 4.5$. But since border rank is an integer, we have $\brk(\det_3) \geq 5$. On the other hand, we have $\brk(\det_3) \leq 5$ by the expression in Section~\ref{upper bounds}, giving us the required conclusion.

\end{proof}

The problem with this argument in characteristic two is that the aforementioned $3 \times 3$ minor has rank $2$ instead of $3$. This only gives that $\trk(\det_3) \geq \brk(\det_3) \geq 4$. Nevertheless, we are able to modify the argument to show that the tensor rank of the $3 \times 3$ determinant is $5$. First we need a simple lemma.

\begin{lemma}
Let $T \in V = V_1 \otimes V_2 \otimes \dots \otimes V_n$. Suppose $\trk(T - S) \geq r$ for every rank $1$ tensor $S \in V$, then we have $\trk(T) \geq r + 1$.
\end{lemma}
\begin{proof}
Suppose $\trk(T) \leq r$, then we have $T = T_1 + \dots + T_k$ with $k \leq r$, where $T_i$ are rank $1$ tensors. Now, take $S = T_1$ to see that $\trk(T - S) \leq k-1 \leq r-1$ contradicting the hypothesis.
\end{proof}

Now, we are ready to prove Theorem~\ref{main}.

\begin{proof} [Proof of Theorem~\ref{main}]
We want to prove that $\trk(\det_3) \geq 5$. By the above lemma, it suffices to prove that $\trk(\det_3 - S) \geq 4$ for every rank $1$ tensor $S$. Observe that $\SL_3$ acts on $K^3 \otimes K^3 \otimes K^3$ by $g \cdot (v_1 \otimes v_2 \otimes v_3) = gv_1 \otimes gv_2 \otimes gv_3$ for $g \in \SL_3$ and $v_i \in K^3$. The action of $g \in \SL_3$ preserves tensor rank and border rank since it is a linear map preserving the set of rank $1$ tensors. There is also an action of the symmetric group on three letters $\Sigma_3$ that permutes the tensor factors. This action too preserves tensor rank and border rank, and further it commutes with the action of $\SL_3$. Thus we have an action of $\SL_3 \times \Sigma_3$ on $K^3 \otimes K^3 \otimes K^3$ that preserves tensor rank and border rank. Further, the tensor $\det_3$ is invariant under this action.

Now, let $S = v_1 \otimes v_2 \otimes v_3$ be a rank $1$ tensor. We want to show $\trk(\det_3 - S) \geq 4$. There are $3$ cases.

\begin{itemize}

\item \textbf{Case 1}: $v_1,v_2,v_3$ are linearly independent. Then w.l.o.g, we can assume $S = \lambda e_1 \otimes e_2 \otimes e_3$, by applying the action of an appropriate $g \in \SL_3$. Now, apply Proposition~\ref{skew-symm} to $T = (\det_3 - \lambda e_1 \otimes e_2 \otimes e_3)$ to get 

$$\brk(T) \geq \frac{1}{2} \rk \left( \arraycolsep=5pt\def\arraystretch{1} \begin{array}{ccc|ccc|ccc}  
0 & 0 & 0 & 0 & 0 & 0 & 0 & 0 & {\color{red} -1} \\
0 & 0 & 0 & 0 & 0 &  1 - \lambda & 0 & 0 & 0 \\
0 & 0 & 0 & 0 & -1 & 0 & 1 & 0 & 0 \\
\hline
0 & 0 & 0 & 0 & 0 & 0 &0 & {\color{red} 1} & 0  \\
0 & 0 & -1 + \lambda & 0 & 0 & 0 &-1 & 0 & 0 \\
0 & {\color{red} 1} & 0 &0 & 0 & 0 & 0 &0 & 0 \\
\hline
0 & 0 & 1 & 0 & -1& 0  & 0 & 0 & 0 \\
0 & 0 & 0  &{\color{red}  1} & 0 & 0 & 0 & 0 & 0 \\
{\color{red} -1} & 0 & 0 & 0 & 0 & 0 & 0 & 0 & 0 \\
\end{array} \right).
$$

Once again observe that the red entries are in a column or row with no other nonzero entries, reducing our computation to a $4 \times 4$ minor. It is easy to check that this gives $\brk(T) \geq 4$ in all characteristic.

\item \textbf{Case 2}: The span $\left< v_1,v_2,v_3 \right>$ is 2-dimensional. In this case, w.l.o.g, can assume $S = e_1 \otimes e_2 \otimes (a e_1 + b e_2)$, by using the action of $\SL_3 \times \Sigma_3$ as in the previous case. Now, apply Proposition~\ref{skew-symm} to $T = (\det_3 -  e_1 \otimes e_2 \otimes (a e_1 + b e_2))$ to get 

$$\brk(T) \geq \frac{1}{2} \rk \left( \arraycolsep=5pt\def\arraystretch{1} \begin{array}{ccc|ccc|ccc}  
0 & 0 & 0 & 0 & 0 & 0 & 0 & 0 & -1 \\
0 & 0 & 0 & -a & -b &  1 & 0 & 0 & 0 \\
0 & 0 & 0 & 0 & -1 & 0 & 1 & 0 & 0 \\
\hline
0 & 0 & 0 & 0 & 0 & 0 &0 & 1& 0  \\
a & b & -1 & 0 & 0 & 0 &-1 & 0 & 0 \\
0 & 1 & 0 &0 & 0 & 0 & 0 &0 & 0 \\
\hline
0 & 0 & 1 & 0 & -1& 0  & 0 & 0 & 0 \\
0 & 0 & 0  &  1 & 0 & 0 & 0 & 0 & 0 \\
 -1 & 0 & 0 & 0 & 0 & 0 & 0 & 0 & 0 \\
\end{array} \right).
$$

Applying the row transformation $R_5 \mapsto R_5 + a R_9 - bR_6$ and the column transformations $C_5 \mapsto  C_5 + bC_6$ and $C_4 \mapsto C_4 + a C_6$, we see that we are back to computing the rank of the matrix in Proposition~\ref{charnot2}, which as we have seen is at least $8$ in all characteristics. Hence $\brk(T) \geq 8/2 = 4$ as required.

\item \textbf{Case 3}: The span $\left< v_1,v_2,v_3 \right>$ is 1-dimensional. Once again, w.l.o.g, $S = \lambda e_1 \otimes e_1 \otimes e_1$. We are reduced to computing the rank of the matrix 
$$
 \left( \arraycolsep=5pt\def\arraystretch{1} \begin{array}{ccc|ccc|ccc}  
0 & 0 & 0 & -\lambda & 0 & 0 & 0 & 0 & -1\\
0 & 0 & 0 & 0 & 0 & 1 & 0 & 0 & 0 \\
0 & 0 & 0 & 0 & -1 & 0 & 1 & 0 & 0 \\
\hline
\lambda & 0 & 0 & 0 & 0 & 0 &0 & 1 & 0  \\
0 & 0 & -1 & 0 & 0 & 0 &-1 & 0 & 0 \\
0 &  1 & 0 &0 & 0 & 0 & 0 &0 & 0 \\
\hline
0 & 0 & 1 & 0 & -1& 0  & 0 & 0 & 0 \\
0 & 0 & 0  &1 & 0 & 0 & 0 & 0 & 0 \\
 -1 & 0 & 0 & 0 & 0 & 0 & 0 & 0 & 0 \\
\end{array} \right).
$$

But again, the row transformations $R_4 \mapsto R_4 + \lambda R_9$ and $R_1 \mapsto R_1 - \lambda R_8$, puts us back to computing the rank of the matrix in Proposition~\ref{charnot2}. The rest of the analysis is as in the previous case. 
\end{itemize}

\end{proof}

While we have successfully computed the tensor rank, the border rank still remains undetermined.

\begin{problem}
What is the border rank of $\det_3$ over an algebraically closed field of characteristic two?
\end{problem}

\section{$5 \times 5$ and $7\times 7$ determinant and permanent tensors}
\label{sec-5x5-7x7}

In this section we study the ranks of the $5 \times 5$ and $7 \times 7$ determinant and permanent tensors.
Assume $K$ is a field of characteristic $0$.
From the results in \cite{DM-explicit}, we know that $\det_3$ has strictly larger tensor rank and border rank than $\per_3$, i.e., 
\[\brk(\per_3) = \trk(\per_3) = 4 < 5 = \trk(\det_3) = \brk(\det_3).\]

We would like to separate $\per_n$ and $\det_n$ for larger $n$. The upper bounds we know for the tensor rank and border rank for $\per_n$ are stronger than the ones we know for $\det_n$. On the other hand the best known lower bounds for both are the same, see \cite{Derksen}. Using Koszul flattenings, we can separate $\det_5$ and $\per_5$. 

\begin{proposition}
We have 
\[13 \leq  \trk(\per_5),\brk(\per_5) \leq 16 < 17 \leq \brk(\det_5),\trk(\det_5) \leq 20.\]

\end{proposition}

\begin{proof}
The upper bounds are due to Glynn and Derksen as mentioned in Section~\ref{upper bounds}. The lower bounds come from applying Proposition~\ref{general-koszul}. This requires finding the rank of a large matrix, which we do with the help of a computer. We omit the details, referring the interested reader to the Python code available at \cite{code}.
\end{proof}

Using the same technique as above, we get the following bounds for the tensor rank and border rank of $\per_7$ and $\det_7$.

\begin{proposition}
We have 
$$
42 \leq \brk(\per_7),\trk(\per_7) \leq 64,
$$
and 
$$
62 \leq \brk(\det_7),\trk(\det_7) \leq 100.
$$
\end{proposition}

Hence, Koszul flattenings are not powerful enough to separate $\per_7$ and $\det_7$. Moreover, we point out that Koszul flattenings are helpful only for finding lower bounds when $n$ is odd.

\section{$3 \times 3$ symmetric determinant and permanent tensors}
\label{sec-symmetric}
A more natural notion is to consider the determinant and permanent as homogeneous polynomials, which in the language of tensors correspond to symmetric tensors. Assuming $K$ is an algebraically closed field of characteristic $0$, we think of homogeneous polynomials in $m$ variables of degree $d$ as elements of $\symPower{d}{K^m} \subset (K^m)^{\otimes d}$.

We define the symmetric determinant tensor
\[ \sdet_n  = \sum_{\sigma \in \Sigma_n}  \sgn(\sigma) x_{1,\sigma(1)} x_{2,\sigma(2)} \dots x_{n,\sigma(n)} \in \symPower{n}{K^{n^2}} \subseteq (K^{n^2})^{\otimes n}, \]
and the symmetric permanent tensor
\[\sperm_n  = \sum_{\sigma \in \Sigma_n}  x_{1,\sigma(1)} x_{2,\sigma(2)} \dots x_{n,\sigma(n)} \in \symPower{n}{K^{n^2}} \subseteq (K^{n^2})^{\otimes n} .\]

\begin{proposition} \label{sdet-lower}
We have $\brk(\sperm_3),\brk(\sdet_3) \geq 14$. 
\end{proposition}

\begin{proof}
We apply Propositon~\ref{general-koszul}. Once again, we omit the details. The interested reader is referred to the Python code available at \cite{code}.
\end{proof}

There is a notion of symmetric rank for a symmetric tensor. 

\begin{definition}
The symmetric rank $\symRank(T)$ of a symmetric tensor $T \in \symPower{d}{K^m}$ is the smallest $r$ such that $T = L_1^{\otimes d} + L_2^{\otimes d} + \dots + L_r^{\otimes d}$, with $L_i \in \symPower{1}{K^m}$. 
\end{definition}

Let $\mathcal{V}_r$ denote the set of symmetric tensors of rank $\leq r$. Once again, $\mathcal{V}_r$ need not be a Zariski closed set.

\begin{definition}
We define symmetric border rank $\symBRank(T)$ of a tensor $T$ as the smallest $r$ such that $T \in \overline{\mathcal{V}_r}$. 
\end{definition}

For a symmetric tensor, we know that the symmetric rank is at least as big as the tensor rank. Hence, we recover a result of Farnsworth in \cite{Farnsworth}.

\begin{corollary} [\cite{Farnsworth}] \label{farn}
We have $\symBRank(\sperm_3),\symBRank(\sdet_3) \geq 14$. 
\end{corollary}

\begin{remark}
It was conjectured by Comon that for a symmetric tensor $T$, we have $\symRank(T) = \trk(T)$ and $\symBRank(T) = \brk(T)$. However, this has recently been proved false, see \cite{Shitov-cex}. In view of this, Proposition~\ref{sdet-lower} is a stronger result than Corollary~\ref{farn}.
\end{remark}

\section{Binary tensors}
\label{sec-binary}
Informally, binary tensors are those whose entries are 0 and 1.
These tensors are interesting because they live in tensor spaces of all characteristics.
A natural question is, for a fixed binary tensor $T$, how does its tensor rank vary as we change the characteristic?

Formally, let $\underline{n} = (n_1, \dots, n_m)$ be a dimension vector, and let $[\underline{n}] \defeq [n_1] \times [n_2] \times \dots \times [n_m]$, where $[k] \defeq \{1,2,\dots, k\}$ for any $k \in \Z_{>0}$.
Consider the free $\Z$-module $V_{\Z} = \Z^{n_1} \otimes \Z^{n_2} \otimes \dots \otimes \Z^{n_m}$. Let $\{e_i\ |\  1 \leq i \leq n_i\}$ denote the standard basis for $\Z^{n^i}$. Then for each $I = (i_1,\dots,i_m) \in [\underline{n}]$, we write $e_I = e_{i_1} \otimes \dots \otimes e_{i_m}$. The set $\{e_I \ |\  I \in [\underline{n}]\}$ form a basis for $V_\Z$.

\begin{definition}
A binary tensor $T \in V_\Z$ is a tensor of the form $T = \sum_{I \in [\underline{n}]} a_Ie_I$, where $a_I = 0$ or $1$. 
\end{definition}

We can compare the ranks of binary tensors across different fields. For any field $K$, we can consider $V_\Z \otimes_\Z K = \K^{n_1} \otimes \K^{n_2} \otimes \dots \otimes \K^{n_m}$. Any tensor $T \in V_\Z$ can be viewed as a tensor in $V_\Z \otimes_\Z K$ by considering $T \otimes_\Z 1$. We will abuse notation, and refer to this tensor by $T$ as well.

\begin{definition}
For a binary tensor $T \in V_\Z$, we define $\trk_K(T)$ as the tensor rank of $T \in V_\Z \otimes_\Z K$. We define $\trk_p(T) = \trk_K(T)$ for any algebraically closed field $K$ of characteristic $p$. 
\end{definition} 

We leave it to the reader to verify that the above definition of $\trk_p$ does not depend on the choice of algebraically closed field. It is easy to find examples of tensors for which $\trk_0(T) \geq \trk_p(T)$. Recall that for matrices over a field, tensor rank coincides with the usual definition of matrix rank.

\begin{proposition}
Let $p$ be a prime, and consider 
\[M = 
\begin{pmatrix}
0 & 1 & 1 & \dots & 1 \\
1 & 0 & 1 & \dots & 1 \\
1 & 1 & \ddots &   & \vdots \\
\vdots & \vdots &  &  \ddots & 1 \\
1 & 1 & \dots & 1 & 0
\end{pmatrix}
\in \M_{p+1,p+1} (\Z) = \Z^{p+1} \otimes \Z^{p+1}.\]
Then we have $\trk_p(M) = p$ and $\trk_0(M) = p+1$. 
\end{proposition}

\begin{proof}
For any field $K$, we can think of $M$ as a linear automorphism of $K^m$. The vector $v = (1,1,\dots,1)^t$ is an eigenvector with eigenvalue $p$, and let $L$ denote the $1$-dimensional space spanned by $v$. The linear map descends to the quotient $K^m/L$, and in this quotient, $M$ acts by the scalar $-1$ as is evident from the fact that $M(e_i) = v - e_i$.

In short, we have that the eigenvalues of $M$ are $p, -1,-1,\dots,-1$, giving us the required conclusion.
\end{proof}

For any matrix $A$, we define $A^{\otimes t}$ the $t$-fold Kronecker product of $A$. Since tensor rank for matrices coincides with the usual rank, and matrix rank is multiplicative w.r.t Kronecker products, we have the following.

\begin{corollary} \label{nolower}
We have $$
\frac{\trk_p(M^{\otimes t})}{\trk_0(M^{\otimes t})} = \left(\frac{p}{p+1}\right)^t.
$$ 
\end{corollary}

In particular, the ratio can be made as small as we wish by taking a large enough power of $t$. Finding a binary tensor which has a larger rank in positive characteristic compared to characteristic $0$ is harder. However, from the main result in this paper, we can deduce the following:

\begin{corollary}
We have
$$
\trk_p(\per_3) = \begin{cases} 5 & \text{if } p = 2 \\
 4 & \text{ otherwise.}
 \end{cases}
$$
\end{corollary}

In other words, we have $\trk_2(\per_3)/\trk_0(\per_3) = 5/4 > 1$. This raises an interesting open question -- is there an upper bound on $\trk_p(T)/\trk_0(T)$?

\begin{problem} \label{last}
Is there a real number $D$ such that $\frac{\trk_p(T)}{\trk_0(T)} \leq D$ for any binary tensor $T$?
\end{problem}

Note that the ratio $\trk_p(T)/\trk_0(T)$ can be as small a positive number as we wish, from Corollary~\ref{nolower}.

\end{document}